\documentclass[9pt]{amsart}

\usepackage{amssymb}

\newtheorem{theorem}{Theorem}
\newtheorem{lemma}{Lemma}
\newtheorem{example}{Example}
\newtheorem{definition}{Definition}

\begin{document}

\title{On Dedekind's problem for complete simple games}
\author{Sascha Kurz}
\author{Nikolas Tautenhahn}
\address{Sascha Kurz\\Department for Mathematics, Physics and Informatics\\University Bayreuth\\Germany}
\email{sascha.kurz@uni-bayreuth.de}
\address{Nikolas Tautenhahn\\Department for Mathematics, Physics and Informatics\\University Bayreuth\\Germany}
\email{nikolas.tautenhahn@uni-bayreuth.de}



\maketitle

\begin{abstract}
We combine the parametric Barvinok algorithm with a generation algorithm for a finite list of suitably chosen
discrete sub-cases on the enumeration of complete simple games, i.~e.{} a special subclass of monotone Boolean functions.
Recently, Freixas et al.{} have proven an enumeration formula for complete simple games with two types of voters. We will provide a shorter proof and an enumeration formula for complete simple games with two shift-minimal winning coalitions.
\end{abstract}

\twocolumn

\section{Introduction}
Consider a yes-no voting system for a set of $n$ voters. The acceptance of a proposal should depend on the subset of ``yea'' voters. In general, the acceptance may be described by a Boolean function. There are several features that we expect to be valid for a voting system. The most natural feature is that a proposal being favored by a set $Y$ of voters will not be rejected if it is accepted with the support of a subset $Y'\subseteq Y$ of the voters. This directly leads us to the class of monotone Boolean functions or simple games.

There are different concepts for the set of required features of a voting system. One is that of a weighted majority game. Here, we are given non-negative voting weights $w_i\in\mathbb{R}_{\ge 0}$ for the voters and a quota $q\in\mathbb{R}_{>0}$. A proposal is accepted if and only if $\sum\limits_{i\in Y}w_i\ge q$, where $Y$ is the set of voters in favor of the proposal. Speaking informally, we can say that for a pair of voters $i$ and $j$ with $w_i<w_j$, voter $i$ has less influence than voter $j$. The exact definition or measurement of ``influence'' is given by a desirability relation introduced by Isbell in \cite{isbell}. There are some cases for such desirability relations which have no realization using non-negative real weights. The class of all total desirability relations is called complete simple games which is a natural extension of weighted majority games. For the other direction, it is well known that each complete simple game can be represented as the intersection of $k$~weighted majority games where the smallest such number $k$ is called the dimension, see \cite{dimension}.

The theory of simple games is a very active area of research. For the broad variety of applications we quote Taylor and Zwicker \cite{0943.91005}: ``Few structures arise in more contexts and lend themselves to more diverse interpretations than do simple games.''

The aim of this paper is to provide an algorithm to determine exact formulas for the number of non-isomorphic complete simple games.

\subsection{Related results}
The enumeration of monotone Boolean functions\footnote{or antichains} is sometimes called ``Dedekind's problem''. By now, the numbers $mb(n)$ of monotone Boolean functions are known for $n\le 8$; they are $3$, $6$, $20$, $168$, $7\,581$, $7\,828\,354$, $2\,414\,682\,040\,998$, and $56\,130\,437\,228\,687\,557\,907\,788$ \cite{1072.06008}. For $n=9$, it is only known that there are more than $10^{42}$ different antichains. Except an asymptotic formula \cite{1072.06008} and a computationally useless exact formula \cite{0632.06020} nearly nothing is known on the values $mb(n)$.

If an additional parameter $k$ is introduced, some exact formulas for the number $mb(n,k)$ of antichains of $\{1,\dots,n\}$ consisting of exactly $k$ subsets can be deduced:
\begin{eqnarray*}
  mb(n,0) &=& 1,\\
  mb(n,1) &=& 2^n,\\
  mb(n,2) &=& 2^n\cdot\frac{2^n - 1}{2}-  3^n + 2^n,\\
  mb(n,3) &=& 2^n\cdot\frac{\left(2^n - 1\right)\left(2^n - 2\right)}{6}-  6^n + 5^n + 4^n - 3^n.
\end{eqnarray*}
In \cite{enum1} Kilibarda and Jovovi\'c gave a general procedure for calculating exact formulas for $mb(n,k)$ where $k$ is fixed and have explicitly listed formulas for $k\le 10$. Roughly speaking, they reduced the enumeration problem for $mb(n,k)$ to the enumeration of all connected bipartite graphs with fixed numbers of vertices and edges, and with a given number of $2$-colorings of a certain type. A generalization of their method is given in \cite{antichains_multisets}.

Unfortunately, the number $k$ of elements of an antichain of $N$ can become quite large. Due to Sperner's theorem, $k$ can vary between $0$ and $n \choose \left\lfloor\frac{n}{2}\right\rfloor$, see e.~g.{} \cite{0868.05001}.

\medskip

The authors of \cite{1151.91021} have enumerated complete simple games with one shift-minimal winning coalition. The number of symmetric complete simple games is a classical result due to May, see e.~g.{} \cite{pre05632618}. An exact formula for the number of complete simple games with two types of voters was proven very recently in \cite{sub_freixas,arxix_freixas}

\medskip

All weighted majority games up to $6$ voters were enumerated in \cite{fishburn_brams,0105.12002}. The enumeration for $7$~voters was done in \cite{0233.94016} and for $8$~voters we refer e.~g.{} to \cite{owen,integer_representation,0841.90134,0205.17805}. Some special classes of complete simple games and weighted majority games were enumerated in \cite{0892.90188,1151.91021}. For asymptotic bounds we refer to \cite{0797.05004}.

\subsection{Our contribution}

We provide a shorter proof (compared to \cite{sub_freixas,arxix_freixas}) for the enumeration of complete simple games with two types of voters and an enumeration formula for complete simple games with two shift-minimal winning coalitions. Our main contribution is an algorithm (based on the parametric Barvinok algorithm) to determine formulas for the number of complete simple games with $t$ types of voters and $r$ shift-minimal winning coalitions in dependence of the number of voters. 

\subsection{Outline of the Paper}
In Section~\ref{sec_complete_simple_games} we provide the basic definitions for complete simple games concluding with a parameterization, see Theorem~\ref{thm_characterization_cs}. Section~\ref{sec_model_lattice_points} is devoted to the modeling process of complete simple games as integer points in polytopes and a very brief introduction into lattice point counting algorithms. Complete simple games with two types of voters will be discussed in Section~\ref{sec_two_types}. Next we describe an approach to determine enumeration formulas for the number of complete simple games before we end with a conclusion in Section~\ref{sec_conclusion}.

\section{Complete simple games}
\label{sec_complete_simple_games}

In this section we will define the crucial objects.

\subsection{Simple games or monotone boolean functions}

Let $N=\{1,\dots,n\}$ be a set of $n$ voters. By $2^N$ we denote the set $\{U\mid U\subset N\}$ of all subsets of $N$. The information whether the support of a proposal by a subset $Y$ of the voters suffices for its acceptance is condensed in a characteristic function $\chi\colon2^N\rightarrow\{0,1\}$.

\begin{definition}
A pair $(N,\chi)$ is called \textbf{simple game} if $\chi$ is a characteristic function of the subsets of $N$ with $\chi(\emptyset)=0$, $\chi(N)=1$, and $\chi(U')\le\chi(U)$ for all $U'\subseteq U$.
\end{definition}

So, each simple game is a monotone Boolean function and except for the all-zero function and the all-one function all monotone Boolean functions are simple games. We will call a subset $U$ of the set of voters $N$ a coalition.

\begin{definition}
A coalition $U\subseteq N$ of a simple game $(N,\chi)$ is called \textbf{winning} iff $\chi(U)=1$ and \textbf{losing} otherwise. By $W$ we denote the set $\left\{U\mid \chi(U)=1\right\}$ of all winning coalitions and by $L$ the set  $\left\{U\mid \chi(U)=0\right\}$ of all losing coalitions.
\end{definition}

Obviously, a simple game can be described by explicitly listing all winning coalitions or all losing coalitions. Since there are $2^n$ subsets of $N$ such a list could become quite large even for rather small values of $n$. So far, the property $\chi(U')\le\chi(U)$ for all $U'\subseteq U$ was not used to compress the data.

\begin{definition}
A coalition $U\subseteq N$ of a simple game $(N,\chi)$ is called a \textbf{minimal winning} coalition iff $\chi(U)=1$ and $\chi(U')=0$ for all
$U'\subsetneq U$. It is called a \textbf{maximal losing} coalition iff $\chi(U)=0$ and $\chi(U')=1$ for all $U\subsetneq U'$. By $\overline{W}$ we
denote the set of all minimal winning coalitions and by $\overline{L}$ the set of all maximal losing coalitions.
\end{definition}

We remark that the knowledge of either $\overline{W}$ or $\overline{L}$ suffices to reconstruct $W$, $L$, and $\chi$. In \cite{complexity} the complexity of changing the representation form of a simple game is studied. Note that the sets $\overline{W}$ and $\overline{L}$ are antichains, i.~e.{} no two elements are subsets of each other.

\begin{example}
\label{ex_first}
\begin{eqnarray*}
  W&=&\Big\{\{1,2\},\{1,2,3\},\{1,2,4\},\{1,2,3,4\},\{1,3\},\\
  &&\{1,4\},\{1,3,4\},\{2,3\},\{2,4\},\{2,3,4\}\Big\}
\end{eqnarray*}
So, there are $10$ winning coalitions. Clearly we have $L=2^N\backslash W$ and $|L|=6$. It is not hard to figure out the sets of
the minimal winning coalitions
$$
  \overline{W}=\Big\{\{1,2\},\{1,3\},\{1,4\},\{2,3\},\{2,4\}\Big\}
$$
and the maximal losing coalitions
$$
  \overline{L}=\Big\{\{1\},\{2\},\{3,4\}\Big\}.
$$
\end{example}

\subsection{Isbell's desirability relation}
As mentioned in the introduction, the monotonicity of simple games is a very weak requirement for voting systems. Now, we define the desirability relation introduced by Isbell in \cite{isbell} (using a different notation):
\begin{definition}
We write $i\sqsupset j$ (or $j \sqsubset i$) for two voters $i,j\in N$ iff we have $\chi\Big(\{i\}\cup U\backslash\{j\}\Big)\le\chi (U)$ for all $\{j\}\subseteq U\subseteq N\backslash\{i\}$ and we abbreviate $i\sqsupset j$, $j\sqsupset i$ by $i\square j$. A pair $(N,\chi)$ is called \textbf{complete simple game} (also called a ``directed game", see \cite{0841.90134}) if it is a simple game and the binary relation $\sqsupset$ is a total preorder, i.~e.{}
\begin{itemize}
  \item[(1)] $i\sqsupset i$ for all $i\in N$,
  \item[(2)] either $i\sqsupset j$ or $j\sqsupset i$ (including ``$i\sqsupset j$ and $j\sqsupset i$'') for all $i,j\in N$, and
  \item[(3)] $i\sqsupset j$, $j\sqsupset h$ implies $i\sqsupset h$ for all $i,j,h\in N$
\end{itemize}
holds. By $\mathbf{cs(n)}$ we denote the number of complete simple games for $n$~voters.
\end{definition}
If $i\sqsupset j$ we say that voter $i$ is at most as desirable as voter $j$ as a coalition partner. If $i\square j$ then voter $i$ is as desirable as voter $j$ as a coalition partner since $\chi(U\cup\{i\})=\chi(U\cup\{j\})$ for all $U\subseteq N\backslash\{i,j\}$. To factor out some symmetry, we partition the set of voters $N$ into subsets $N_1,\dots,N_t$ such that we have $i\square j$ for all $i,j\in N_h$, where $1\le h\le t$, and $i\square j$
implies the existence of an integer $1\le h\le t$ with $i,j\in N_h$. So in some sense the sets $N_i$ cover equally desirable voters. In Example \ref{ex_first} we have $t=2$, $N_1=\{1,2\}$ and $N_2=\{3,4\}$.

We would like to remark that this is also known in the field of Boolean algebra as Winder's preorder, see e.~g.{} \cite{0243.94014,phd_winder,0207.02101}.

Given these classes of equally desirable voters we can further compress the information on the minimal winning coalitions or maximal losing coalitions:
\begin{definition}
\label{def_winning_vector}
Let $(N,\chi)$ be a complete simple game and $N_i$ be the classes of equally desirable voters for $1\le i\le t$. We call a
vector $\widetilde{m}:=\begin{pmatrix}m_1&\dots&m_t\end{pmatrix}$, where $0\le m_i\le \left|N_i\right|$ for $1\le i\le t$, a
\textbf{winning coalition} iff $\chi(U)=1$, where $U$ is an arbitrary subset of $N$ containing exactly $m_i$ elements of $N_i$
for $1\le i\le t$. Analogously, we call such a vector a \textbf{losing coalition} iff $\chi(U)=0$, where $U$ is an arbitrary
subset of $N$ containing exactly $m_i$ elements of $N_i$ for $1\le i\le t$.
\end{definition}

Due to the definition of the sets $N_i$ the value of $\chi(U)$ does only depend on $\widetilde{m}$ and not on $U$. In Example \ref{ex_first} the winners are given by
$$
  \Big\{
  \begin{pmatrix}1&1\end{pmatrix},
  \begin{pmatrix}1&2\end{pmatrix},
  \begin{pmatrix}2&0\end{pmatrix},
  \begin{pmatrix}2&1\end{pmatrix},
  \begin{pmatrix}2&2\end{pmatrix}
  \Big\}
$$
and the losers are given by
$$
  \Big\{
  \begin{pmatrix}0&0\end{pmatrix},
  \begin{pmatrix}0&1\end{pmatrix},
  \begin{pmatrix}0&2\end{pmatrix},
  \begin{pmatrix}1&0\end{pmatrix}
  \Big\}.
$$

To factor out remaining symmetries, we require $i\sqsubset j$ for $i\in N_{h_1}$, $j\in N_{h_2}$, where $1\le h_1<h_2\le t$, and
$$N_i=\Big\{\sum_{j=1}^{i-1}\left|N_j\right|+1,\dots,\sum_{j=1}^{i}\left|N_j\right|\Big\},$$
which can always be achieved by rearranging the voters of a given complete simple game. With this, it suffices to know the cardinalities $n_i:=\left|N_i\right|$ instead of the explicit sets $N_i$.

\begin{definition}
  We call a pair $\left(\widetilde{n},\chi\right)$, where $\widetilde{n}=\begin{pmatrix}n_1&\dots&n_t\end{pmatrix}$ with
  $n_i\in\mathbb{N}_{>0}$ and $\sum_{i=1}^t n_i=:n$, a \textbf{complete simple game} if $(N,\chi)$ is a complete simple game
  and the $N_i$ are the classes of equally desirable voters for $1\le i\le t$, where $N$ and the $N_i$ are given as stated above.
  By $\mathbf{cs(n,t)}$, we denote the number of complete simple games with $t$ equivalence classes for $n$ voters.
\end{definition}

To carry over the concept of minimal winning coalitions and maximal losing coalitions to vectors, we need a partial ordering:

\begin{definition}
  \label{def_smaller_vector}
  For two vectors $\widetilde{a}=\begin{pmatrix}a_1&\dots&a_t\end{pmatrix}$ and
  $\widetilde{b}=\begin{pmatrix}b_1&\dots&b_t\end{pmatrix}$ we write $\widetilde{a}\preceq \widetilde{b}$ if and only if
  we have
  \[
    \sum_{i=1}^{k} a_i \le \sum_{i=1}^{k} b_i
  \]
  for all $1\le k\le t$. For $\widetilde{a}\preceq \widetilde{b}$ and $\widetilde{a}\neq \widetilde{b}$ we use
  $\widetilde{a}\prec\widetilde{b}$ as an abbreviation. If neither $\widetilde{a}\preceq \widetilde{b}$ nor
  $\widetilde{b}\preceq \widetilde{a}$ holds we write $\widetilde{a}\bowtie \widetilde{b}$.
\end{definition}
In words, we say that $\widetilde{a}$ is smaller than $\widetilde{b}$ if $\widetilde{a}\prec\widetilde{b}$ and that $\widetilde{a}$ and $\widetilde{b}$ are incomparable if $\widetilde{a}\bowtie \widetilde{b}$. For some background on this partial ordering in the literature, we refer the interested reader to \cite{integer_representation}. Like any partial order, $\preceq$ can be depicted by a Hasse diagram or a directed graph. In Figure \ref{fig_hasse_2_2} we have drawn the Hasse diagram for Example \ref{ex_first}. In such a Hasse diagram or directed graph, we have $\widetilde{a}\preceq\widetilde{b}$ if there is a (possibly empty) directed path from $\widetilde{a}$ to $\widetilde{b}$. We have $\widetilde{a}\bowtie\widetilde{b}$ if there is neither a directed path from $\widetilde{a}$ to $\widetilde{b}$ nor from $\widetilde{b}$ to $\widetilde{a}$.

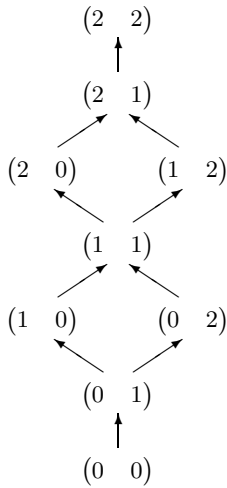
\begin{figure}[htp]
  \begin{center}
    \setlength{\unitlength}{1cm}
    \begin{picture}(3,7)
      \put(1,0){$\begin{pmatrix}0&0\end{pmatrix}$}
      \put(1,1){$\begin{pmatrix}0&1\end{pmatrix}$}
      \put(0,2){$\begin{pmatrix}1&0\end{pmatrix}$}
      \put(2,2){$\begin{pmatrix}0&2\end{pmatrix}$}
      \put(1,3){$\begin{pmatrix}1&1\end{pmatrix}$}
      \put(0,4){$\begin{pmatrix}2&0\end{pmatrix}$}
      \put(2,4){$\begin{pmatrix}1&2\end{pmatrix}$}
      \put(1,5){$\begin{pmatrix}2&1\end{pmatrix}$}
      \put(1,6){$\begin{pmatrix}2&2\end{pmatrix}$}
      \put(1.5,0.4){\vector(0,1){0.45}}
      \put(1.3,1.4){\vector(-3,2){0.65}}
      \put(1.7,1.4){\vector(3,2){0.65}}
      \put(0.7,2.4){\vector(3,2){0.65}}
      \put(2.3,2.4){\vector(-3,2){0.65}}
      \put(1.3,3.4){\vector(-3,2){0.65}}
      \put(1.7,3.4){\vector(3,2){0.65}}
      \put(0.7,4.4){\vector(3,2){0.65}}
      \put(2.3,4.4){\vector(-3,2){0.65}}
      \put(1.5,5.4){\vector(0,1){0.45}}
    \end{picture}
    \label{fig_hasse_2_2}
    \caption{The Hasse diagram for $\preceq$ on $(2\,\,\,2)$.}
  \end{center}
\end{figure}

With Definition \ref{def_smaller_vector} at hand, we can define:
\begin{definition}
A vector $\widetilde{m}:=\begin{pmatrix}m_1&\dots&m_t\end{pmatrix}$ in a complete simple game $\Big(\begin{pmatrix}n_1&\dots&n_t\end{pmatrix},\chi\Big)$ is a \textbf{minimal winning coalition}\footnote{Sometimes we more precisely speak of shift-minimal winning coalitions.} iff
$\widetilde{m}$ is a winning coalition and every coalition $\widetilde{m}'\prec\widetilde{m}$ is losing. Analogously,
a coalition $\widetilde{m}$ is a \textbf{maximal losing coalition}\footnote{Similarly we also speak of shift-maximal losing coalitions.} iff $\widetilde{m}$ is a losing coalition and every
coalition $\widetilde{m}'\succ\widetilde{m}$ is winning.
\end{definition}

The minimal winning coalitions of Example \ref{ex_first} are given by
$$
  \Big\{\begin{pmatrix}1&1\end{pmatrix}\Big\}
$$
and the maximal losing coalitions are given by
$$
  \Big\{\begin{pmatrix}1&0\end{pmatrix},\begin{pmatrix}0&2\end{pmatrix}\Big\},
$$
which both are quite short lists. Thus, a complete simple game can be represented by a vector $$\widetilde{n}=\begin{pmatrix}n_1&\dots&n_t\end{pmatrix}$$ and a matrix $$\mathcal{M}=\begin{pmatrix}m_{1,1}&m_{1,2}&\dots&m_{1,t}\\m_{2,1}&m_{2,2}&\dots&m_{2,t}\\
\vdots&\ddots&\ddots&\vdots\\m_{r,1}&m_{r,2}&\dots&m_{r,t}\end{pmatrix}=
\begin{pmatrix}\widetilde{m}_1\\\widetilde{m}_2\\\vdots\\\widetilde{m}_r\end{pmatrix}$$ of all
minimal winning coalitions $\widetilde{m}_i$. In Example \ref{ex_first} we have $\widetilde{n}=\begin{pmatrix}2&2\end{pmatrix}$ and $\mathcal{M}=\begin{pmatrix}1&1\end{pmatrix}$. In general, this is a very compact representation of a complete simple game
and we will use it in the remaining part of the paper.

In order to state an isomorphism-free parameterization of complete simple games, we need another total ordering, the so-called lexicographic ordering:
\begin{definition}
  \label{def_lexicographic}
  For two vectors $\widetilde{a}=\begin{pmatrix}a_1&\dots&a_t\end{pmatrix}$ and
  $\widetilde{b}=\begin{pmatrix}b_1&\dots&b_t\end{pmatrix}$ we write $\widetilde{a}\gtrdot \widetilde{b}$
  if and only if there is an integer $1\le h\le t$ such that $a_i=b_i$ for $i<h$ and $a_h>b_h$.
\end{definition}

\noindent
and a specification when we consider two complete simple games to be isomorphic:
\begin{definition}
  \label{def_isomorphism}
  Two complete simple games $\left(N_1,\chi_1\right)$ and $\left(N_2,\chi_2\right)$ are called \textbf{isomorphic} iff there is
  a bijection $\varphi:N_1\rightarrow N_2$ such that 
  $$
    \chi_1\Big(\Big\{u_1,\dots,u_k\Big\}\Big)=\chi_2\Big(\Big\{\varphi\left(u_1\right),\dots,\varphi\left(u_k\right)\Big\}\Big)
  $$
  for all subsets $\left\{u_1,\dots,u_k\right\}\subseteq N_1$.
\end{definition}

\subsection{A parameterization of complete simple games with t~types of voters}
The crucial characterization theorem for complete simple games using vectors as coalitions and the partial order~$\preceq$ was given by Carreras and Freixas in \cite{complete_simple_games}:
\begin{theorem}
  \label{thm_characterization_cs}

  \vspace*{0mm}

  \noindent
  \begin{itemize}
   \item[(a)] Given are a vector $$\widetilde{n}=\begin{pmatrix}n_1&\dots&n_t\end{pmatrix}\in\mathbb{N}_{>0}^t$$ and a matrix
              $$\mathcal{M}=\begin{pmatrix}m_{1,1}&m_{1,2}&\dots&m_{1,t}\\m_{2,1}&m_{2,2}&\dots&m_{2,t}\\
              \vdots&\ddots&\ddots&\vdots\\m_{r,1}&m_{r,2}&\dots&m_{r,t}\end{pmatrix}=
              \begin{pmatrix}\widetilde{m}_1\\\widetilde{m}_2\\\vdots\\\widetilde{m}_r\end{pmatrix}$$
              satisfying the following properties
              \begin{itemize}
               \item[(i)]   $0\le m_{i,j}\le n_j$, $m_{i,j}\in\mathbb{N}$ for $1\le i\le r$, $1\le j\le t$,
               \item[(ii)]  $\widetilde{m}_i\bowtie\widetilde{m}_j$ for all $1\le i<j\le r$,
               \item[(iii)] for each $1\le j<t$ there is at least one row-index $i$ such that
                            $m_{i,j}>0$, $m_{i,j+1}<n_{j+1}$ if $t>1$ and $m_{1,1}>0$ if $t=1$, and
               \item[(iv)]  $\widetilde{m}_i\gtrdot \widetilde{m}_{i+1}$ for $1\le i<r$.
              \end{itemize}
              Then, there exists a complete simple game $(N,\chi)$ associated to $\left(\widetilde{n},\mathcal{M}\right)$.
   \item[(b)] Two complete games $\left(\widetilde{n}_1,\mathcal{M}_1\right)$ and $\left(\widetilde{n}_2,\mathcal{M}_2\right)$
              are isomorphic if and only if $\widetilde{n}_1=\widetilde{n}_2$ and $\mathcal{M}_1=\mathcal{M}_2$.
  \end{itemize}
\end{theorem}

\noindent
In such a vector/matrix representation of a complete simple game, the number of voters $n$ is determined by $n=\sum\limits_{i=1}^t n_i$. Although Theorem \ref{thm_characterization_cs} looks technical at first glance, the necessity of the required properties can be explained easily. Obviously,  $n_j\ge 1$ and $0\le m_{i,j}\le n_j$ must hold for $1\le i\le r$, $1\le j\le t$. If $\widetilde{m}_i\preceq\widetilde{m}_j$ or $\widetilde{m}_i\succeq\widetilde{m}_j$ then we would have $\widetilde{m}_i=\widetilde{m}_j$ or either $\widetilde{m}_i$ or $\widetilde{m}_j$ cannot be a minimal winner. If for a column-index $1\le j<t$ we have $m_{i,j}=0$ or $m_{i,j+1}=n_{j+1}$ for all $1\le i\le r$, then we can check whether we have $g\square h$ for all $g\in N_j$, $h\in N_{j+1}$, which is a contradiction to the definition of the classes $N_j$ and therefore also for the numbers $n_j$. Obviously, a complete simple game does not change if two rows of the matrix $\mathcal{M}$ are interchanged. Thus, we can require an arbitrary ordering of the rows.

In the following we will denote by $cs(n,t,r)$ the number of complete simple games with $t\le n$ equivalence classes of the $n$ voters and $r$ shift-minimal winning coalitions.

\begin{table}[htp]
\begin{center}
\begin{tabular}{rrrrrrrrrrrrrrrr}
        \hline
        \!$\mathbf{n}$\!&\!\!1\!\!&\!\!2\!\!&\!\!3\!\!&\!\!4\!\!&\!\!5\!\!&\!\!6\!\!&\!\!7\!\!&\!\!8\!\!&\!\!9\!\!&\!\!10\!\!
        &\!\!11\!\!&\!\!12\!\!&\!\!13\!\!&\!\!14\!\!&\!\!15\!\!\\
        \!$\mathbf{max}$\!&\!\!1\!\!&\!\!1\!\!&\!\!2\!\!&\!\!2\!\!&\!\!3\!\!&\!\!5\!\!&\!\!8\!\!&\!\!14\!\!&\!\!23\!\!&
        \!\!40\!\!&\!\!70\!\!&\!\!124\!\!&\!\!221\!\!&\!\!397\!\!&\!\!722\!\!\\
        \hline
\end{tabular}
\caption{The maximum number of shift-minimal winning coalitions.}
\label{table_max_r}
\end{center}
\end{table}

We would like to remark that the maximum number $r$ of shift-minimal winning coalitions can become quite large, see Table~\ref{table_max_r} for the very first numbers. Their number indeed equals
\begin{equation*}
  \max_{1\le h\le n,\,k} \left|\left\{\left(a_1,\dots,a_h\right)\mid
  1\le a_1<\dots<a_h\le n,\,\sum_{i=1}^h a_i=k\right\}\right|,
\end{equation*}
see \cite{0841.90134}.

\section{Complete simple games as lattice points in a polytope}
\label{sec_model_lattice_points}

Now, we have a look at the vector/matrix representation of a complete simple game from a different point of view. In principle, for given parameters $n$, $t$, and $r$, a complete simple game can be described by $(r+1)t$ integers fulfilling certain conditions. If these conditions can be reformulated as linear inequalities, complete simple games with parameters $n$, $t$, and $r$ are in bijection to lattice points of a certain rational polytope, i.~e.{} a polytope where the coordinates of all corners are rational numbers. There is a profound theory on the enumeration of lattice points. The most important result is that for fixed dimension the enumeration can be done in polynomial time in terms of the input size. We will go into the details in the next section.

We would like to remark that for $t=1$ only $r=1$ is possible and the requirements reduce to $1\le m_{1,1}\le n_1=n$. Also for $t=2$ one can easily give a more compact formulation for the requirements in Theorem \ref{thm_characterization_cs}. A complete description of the possible values $n_1,n_2,m_{1,1},m_{1,2}$ corresponding to a complete simple game with parameters $n$, $t=2$, and $r=1$ is given by
\begin{eqnarray}
  && 1\le n_1\le n-1,\label{compact_ilp_2_1}\\
  && n_1+n_2=n,\nonumber\\
  && 1\le m_{1,1} \le n_1,\nonumber\\
  && 0\le m_{1,2}\le n_2-1.\nonumber
\end{eqnarray}
For $t=2$ and $r\ge 2$ such a complete and compact description is given by
\begin{eqnarray}
  && 1\le n_1\le n-1,\label{compact_ilp_2_ge_2}\\
  && n_1+n_2 = n,\nonumber\\
  && m_{i,1}\ge m_{i+1,1}+1\quad\quad\quad\quad\quad\quad\quad\quad\,\forall1\le i\le r-1,\nonumber\\
  && m_{i,1}+m_{i,2}+1\le m_{i+1,1}+m_{i+1,2}\quad\forall1\le i\le r-1.\nonumber
\end{eqnarray}
Thus, for $t\le 2$ we are able to describe the feasible values of the $n_i$ and the $m_{i,j}$ by linear constraints.

\subsection{Complete simple games as integer points in the stable set polytope}

We can easily describe complete simple games as cliques in a suitable graph as follows: Let $\mathcal{G}_n$ be a graph
consisting of the vertices $\{0,1\}^n\backslash 0$ where two vertices $v_i$, $v_j$ are joined by an edge if either $v_i\preceq v_j$
or $v_i\succeq v_j$. The number of (non-trivial) stable sets\footnote{or independent sets} of $\mathcal{G}_n$ equals the number $cs(n)$ of complete simple games with $n$ voters. Thus, we may describe the set of complete simple games as integer points of the stable set polytope of $\mathcal{G}_n$. By counting the vertices of $\mathcal{G}_n$, we obtain
\begin{equation}
  \sum\limits_{t=1}^n cs(n,t,1)=2^n-1.
\end{equation}
\begin{theorem}
\begin{eqnarray*}
  \!\!\!\!\!\!&&\sum\limits_{t=1}^n cs(n,t,2)=2\cdot\left(4^n+2^n\right)\cdot\\
  \!\!\!\!\!\!&&\left(6\cdot\frac{8n-4n^2-3}{n(n-3)}{2n-5\choose n-4}+\frac{2n^2+3n-2}{(n+1)(n-2)}{2n-3\choose n-3}\right)
\end{eqnarray*}
\end{theorem}
\begin{proof}
We count the number of edges $\{u,v\}$ of the complement $\overline{\mathcal{G}_n}$ of $\mathcal{G}_n$, where $u\gtrdot v$.
Suppose that the first $i$~coordinates of $u$ and $v$ coincide, $u_{i+1}=1$, and $v_{i+1}=0$. Now let $n-j-1$ be the first index such that $\sum_{h=1}^{n-j-1} v_h>\sum_{h=1}^{n-j-1} u_h$. Thus we have $u_{n-j-1}=0$,  $v_{n-j-1}=1$, and $\sum_{h=1}^{n-j-2} u_h=\sum_{h=1}^{n-j-2} v_h$. The remaining $j$ coordinates are arbitrary. This gives $\sum_{t=1}^n cs(n,t,2)=$
$$
  \sum\limits_{i=0}^{n-3}\sum\limits_{j=0}^{n-3-i} 
  \!2^i f(n-i-j-1)\cdot 4^j
  =\sum\limits_{i=0}^{n-3} 2^i\left(2^{i+1}-1\right)f(n-i-1),
$$
where $f(k)$ counts the number of sequences $u',v'\in\{0,1\}^{k}$ such that $\sum_{h=1}^{i} u'_h\ge \sum_{h=1}^{i} v'_h$
for all $1\le i\le k$ and $\sum_{h=1}^{k} u'_h= \sum_{h=1}^{k} v'_h$. If we consider only the coordinates of $u'$ and $v'$ which
differ, we obtain Dyck paths, which are counted by the Catalan numbers. The remaining coordinates are equal but arbitrary. Thus we get
$$
  f(k)=\sum\limits_{i=1}^{\left\lfloor\frac{k}{2}\right\rfloor}
  \frac{1}{i+1}{2i\choose i}\cdot{k-1\choose 2i-1}\cdot 2^{k-2i}
  =\frac{4}{k+2}{2k-1\choose k-2}
$$
and apply a last time a computer algebra package like \texttt{Maple 11} for the remaining
summation.
\end{proof}

For a fixed number of voters $n$ one can apply a software package like \texttt{cliquer} \cite{cliquer,1019.05054} in order to exhaustively generate all cliques of $\overline{\mathcal{G}_n}$. see Table~\ref{table_csg_cliquer}. The corresponding values of $t$ and $r$ can be easily determined for each clique separately.

\begin{table}[htp]
\begin{center}
\begin{tabular}{crrrrrrrrr}
\hline
$\!\!\!\mathbf{n}\!\!\!$&$\!\!\!1\!\!\!$&$\!\!\!2\!\!\!$&$\!\!\!3\!\!\!$&$\!\!\!4\!\!\!$&$\!\!\!5\!\!\!$&$\!\!\!6\!\!\!$&
$\!\!\!7\!\!\!$&$\!\!\!8\!\!\!$&$\!\!\!9\!\!\!$\\
$\!\!\!\mathbf{cs(n)}\!\!\!$&$\!\!\!1\!\!\!$&$\!\!\!3\!\!\!$&$\!\!\!8\!\!\!$&$\!\!\!25\!\!\!$&$\!\!\!117\!\!\!$&$\!\!\!1171\!\!\!$&
$\!\!\!44313\!\!\!$&$\!\!\!16175188\!\!\!$&$\!\!\!284432730174\!\!\!$\\
$\!\!\!$\textbf{time}$\!\!\!$&&&&&&&$\!\!\!$0.02~s$\!\!\!$&$\!\!\!$1.23~s$\!\!\!$&$\!\!\!$44:01~m$\!\!\!$\\
\hline
\end{tabular}
\caption{Complete simple games for $\mathbf{n}$ voters.}
\label{table_csg_cliquer}
\end{center}
\end{table}

\subsection{Counting integer points in polytopes}
\label{subsec_counting_integer_points}

In this section we give a very short introduction into the theory of counting integer points in polytopes. For a report on the state-of-the-art we refer to \cite{enumerator} and for an accessible survey on applications of lattice-point enumeration problems we refer to \cite{1093.52006}.

In 1983 Lenstra has achieved a major breakthrough by providing an algorithm which allows to decide whether a given rational polytope $P=\left\{x\in\mathbb{R}^m\mid Ax\le b\right\}$, i.~e.{} where the entries of $A$ and $b$ are rational numbers, contains a lattice point in polynomial time for every fixed dimension $m$ \cite{0524.90067}. The next breakthrough came in 1994 when Barvinok introduced an algorithm for determining the exact number of lattice points in a rational polytope, that runs in polynomial time for every fixed dimension $m$ \cite{0821.90085}.

Here so-called generating functions were utilized for the counting process. We give a brief example in dimension one. Let $[a,b]$ be an interval with rational endpoints $a$ and $b$. The corresponding polytope is given by $$P=\left\{x\in\mathbb{R}^1\mid x\le b,\, -x\le -a\right\}.$$ The polynomial
\begin{equation}
\chi(P,x)=\sum\limits_{i\in P\cap\mathbb{Z}^1} x^i
\end{equation}
is some kind of a characteristic function for the lattice points of $P$, i.~e.{} every lattice point corresponds to a monomial. Using rational functions, even very large polytopes can be written very compactly. Let us have an example: For $a=5$ and $b=2945$ we have $\chi(P,x)=\sum\limits_{i=5}^{2945} x^i=x^5\cdot \frac{x^{2941}-1}{x-1}$ using the geometric series. By evaluating $\chi(P,x)$ at $x=1$ we obtain the number of lattice points in $P$. For higher dimensions polynomials and rational functions in several variables are used. For the mathematical and algorithmic details we refer e.~g.{} to \cite{1114.52013,1137.52303}.

By computing the so-called Ehrhart series of a rational polytope $P$, one can even determine a formula for the number of lattice points in the dilation $nP$. Parts of the well established theory on counting integer points in rational polytopes are rediscovered by researchers in social choice theory. For applications in this area and an explanation of the basic principles we refer e.~g.{} to \cite{1149.91028,1141.91379}.

Here, we will give an exemplary application and example from \cite{1149.91028}. At first we need some definitions.
\begin{definition}
A \textbf{rational periodic number} $f(p)$ is a function $\mathbb{Z}\rightarrow\mathbb{Q}$, such that there is a \textbf{period} $q$ such that $f(p)=f(p')$ whenever $p\equiv p'\pmod q$.
\end{definition}
Ehrhart used a list of $q$ rational numbers 
$$
  \Big[f_1,\dots,f_q\Big]_p
$$
enclosed in square brackets to represent the periodic number $f(p)=f_i$ for $p\equiv i\pmod q$ \cite{0337.10019}. For a more compact representation of periodic numbers we refer to \cite{enumerator}.
\begin{definition}
A quasi-polynomial of degree $m$ is a function
$$
  f(n)=a_d(n)n^d+\dots a_1(n)n+a_0(n),
$$
where the $a_i$ are periodic numbers. The least common multiple of the periods of the coefficients $a_i$ is called the period of $f$.
\end{definition}
Ehrhart's main theorem \cite{0337.10019} is:
\begin{theorem}
Let $P\subseteq \mathbb{Q}^m$ be a rational polytope. The number of lattice points in the dilations $nP$ with $n\in\mathbb{N}$ is given by a degree-$m$ quasi-polynomial where the period is a divisor of the least common multiple of the denominators of the vertices of $P$.
\end{theorem}
So let us consider the polytope $P\in\mathbb{R}^2$ given by the inequalities $x_1+x_2\le 3$, $2x_1\le 5$, $x_1\ge 0$, and $x_2\ge 0$. The vertices of $P$ are given by $(0,0)$, $\left(\frac{5}{2},0\right)$, $(0,3)$, and $\left(\frac{5}{2},\frac{1}{2}\right)$. Thus the number $f(P,n)$ of lattice points of the polytope $nP=\left\{x\in\mathbb{R}^2\mid x_1+x_2\le 3n,\,2x_1\le 5n,\,x_1\ge 0,\,x_2\ge 0\right\}$ is a quasi-polynomial of degree $2$ with period $2$. Indeed we have
$$
  f(P,n)=\frac{35}{8}n^2+\left[\frac{17}{4},4\right]_n\cdot n+\left[1,\frac{5}{8}\right]_n.
$$

This technique was extended to so-called parametric polytopes, see e.~g.{} \cite{1017.05008,enumerator}. To put it in a nutshell, 
one can say: if $P\subseteq\mathbb{R}^{m_1+m_2}$ is a polyhedron such that for every admissible integer vector $p\in\mathbb{Z}^{m_2}$ of the last coordinates, $P$ with fixed last coordinates is a polytope, then one can compute a multivariate quasi-polynomial for the number of integer points in $P$ in dependence of the parameter vector $p$. 

Thus, one aims to formulate the set of feasible points of a counting problem via a rational polytope $P$.

\section{Complete simple games for two types of voters}
\label{sec_two_types}

Using the linear systems of inequalities~(\ref{compact_ilp_2_1}), (\ref{compact_ilp_2_ge_2}), the theory from Subsection~\ref{subsec_counting_integer_points}, and the software package \texttt{barvinok} \cite{enumerator}, we were able to determine formulas for $cs(n,2,r)$, where $r\le 10$. (It took less than $8$~hours of computation time to compute $cs(n,2,8)$, $63$~hours for $cs(n,2,9)$, and $539$~hours for $cs(n,2,10)$.)

These calculations are doable for small $r$; however, for this special case of $t=2$ one can determine a general formula as follows. 
\begin{lemma}
Each complete simple game, given by $\begin{pmatrix}n_1&n_2\end{pmatrix}$, $\begin{pmatrix}
m_{0,1}&m_{0,2}\\
\vdots&\vdots\\
m_{r,1}&m_{r,2}
\end{pmatrix}$,
with two types of voters, i.~e.{} $t=2$, and $r+1\ge 2$ shift-minimal winning coalitions can be written as
$$
\begin{array}{c}
  \begin{pmatrix}r+\sum\limits_{j=0}^{r-1}x_j+y_{r+1}+z_1&2r+\sum\limits_{j=0}^{r-1}x_j+\sum\limits_{j=0}^{r}y_j+z_2\end{pmatrix}\\
  \begin{pmatrix}
    r+\sum\limits_{j=0}^{r-1}x_j+y_{r+1}&0+y_0\\
    r-1+\sum\limits_{j=1}^{r-1}x_j+y_{r+1}&2+x_0+y_0+y_1\\
    \vdots&\vdots\\
    r-i+\sum\limits_{j=i}^{r-1}x_j+y_{r+1}&2i+\sum\limits_{j=0}^{i-1}x_j+\sum\limits_{j=0}^{i}y_j\\
    \vdots&\vdots\\
    1+x_{r-1}+y_{r+1}&2r-2+\sum\limits_{j=0}^{r-2}x_j+\sum\limits_{j=0}^{r-1}y_j\\
    0+y_{r+1}&2r+\sum\limits_{j=0}^{r-1}x_j+\sum\limits_{j=0}^{r}y_j
  \end{pmatrix}
\end{array}\!\!\!,
$$
where $x_0,\dots,x_{r-1},y_0,\dots,y_{r+1},z_1,z_2$ are non-negative integers
fulfilling
\begin{equation}
  \sum\limits_{i=0}^{r-1}2x_i\,+\,\sum\limits_{i=0}^{r+1}y_i\,+\,z_1\,+\,z_2\,=\,n-3r.
  \label{eq_two_types_param}
\end{equation}
\end{lemma}
\begin{proof}
For one direction, we only have to check the conditions of Theorem~\ref{thm_characterization_cs} or directly check the
system of linear inequalities~(\ref{compact_ilp_2_ge_2}).

For the other direction, we state that one can recursively determine the $x_h$, $y_i$, and $z_j$ via
\begin{eqnarray*}
  y_0 &=& m_{0,2},\\
  y_{r+1}&=& m_{r,1},\\
  x_i &=& m_{i,1}-m_{i+1,1}-1\quad\quad\,\quad\quad\text{ for }i=r-1,\dots,0,\\
  y_i &=&m_{i,2}-m_{i-1,2}-2-x_{i-1}\quad\text{ for }i=1,\dots,r,\\
  z_1 &=&n_1-r-\sum\limits_{j=0}^{r-1}x_j-y_{r+1},\text{ and}\\
  z_2 &=&n_2-2r-\sum\limits_{j=0}^{r-1}x_j-\sum\limits_{j=0}^{r}y_j.
\end{eqnarray*}
Verifying $x_h,y_i,z_j\ge 0$ finishes the proof.
\end{proof}

Counting the number of solutions of Equation~(\ref{eq_two_types_param}) and determining the maximum possible $r$ gives:
\begin{lemma}
  For $r\ge 2$ we have
  \begin{equation*}
    cs(n,2,r)=\sum\limits_{i=0}^{\left\lfloor\frac{n-3r+3}{2}\right\rfloor} {i+r-2\choose r-2}{n-2r-2i+5\choose r+2}.
  \end{equation*}
\end{lemma}
For $r=1$ we have
\begin{equation}
  cs(n,2,1)=\frac{n^3-n}{6}.
\end{equation}
The corresponding ordinary generating functions are given by $\frac{x^2}{(1-x)^4}$ for $r=1$ and by $\frac{x^{3r-3}}{(1-x)^{r+3}(1-x^2)^{r-1}}$ for $r\ge 2$.

\begin{theorem}
  For all $n\ge 1$ the number of complete simple games with $t=2$ equivalence classes of voters is given by
  $$cs(n,2)=Fibonacci(n+6)\,-\,\left(n^2+4n+8\right),$$
  where $Fibonacci(0)=0$, $Fibonacci(1)$, and $Fibonacci(i)=Fibonacci(i-1)+Fibonacci(i-2)$ for $i\ge 2$.
\end{theorem}
\begin{proof}
  Summing up the generating functions for $cs(n,2,r)$ yields
  \begin{eqnarray*}
  &&\frac{x^2}{(1-x)^4}+\sum_{r=2}^{\infty} \frac{x^{3r-3}}{(1-x)^{r+3}(1-x^2)^{r-1}}\\
  &=&\frac{x^2}{(1-x)^4}+\frac{x^3}{(1-x)^5(1-x^2)}\sum_{r=0}^{\infty}\left(\frac{x^3}{(1-x)(1-x^2)}\right)^r\\
  &=&\frac{x^2(1-x)(1-x^2)+x^3\cdot\frac{1}{1-\frac{x^3}{(1-x)(1-x^2)}}}{(1-x)^5(1-x^2)}\\
  &=&\frac{x^2(1+x)}{(1-x)^3(1-x-x^2)}\\
  &=&\frac{21x^2+13x^3}{1-x-x^2}-\frac{20x^2-31x^3+13x^4}{(1-x)^3}.
\end{eqnarray*}
Since $\frac{x}{1-x-x^2}$ is the generating function for the Fibonacci numbers $Fibonacci(n)$ and $\frac{x}{(1-x)^3}$ is the generating function for the sequence $\frac{n(n+1)}{2}$ we conclude $cs(n,2)$
\begin{eqnarray*}
  &=&21\cdot Fibonacci(n-1)+13\cdot Fibonacci(n-2)\\
  &&-\frac{20n(n-1)-31(n-1)(n-2)+13(n-2)(n-3)}{2}\\
  &=& Fibonacci(n+6)-\left(n^2+4n+8\right)
\end{eqnarray*}
iteratively using the equation $Fibonacci(i)=Fibonacci(i-1)+Fibonacci(i-2)$ for $i\ge 2$.
\end{proof}

\section{Formulas for the number of complete simple games}
\label{sec_main}

In this section we will give a general algorithm in order to determine formulas for the number $cs(n,t,r)$ of complete simple games with respect to the number $t$ of equivalence classes of voters and the number $r$ of shift-minimal winning coalitions. The main tool  is the parametric Barvinok algorithm, see Subsection~\ref{subsec_counting_integer_points} or \cite{enumerator}. It is indeed possible to express the conditions from Theorem~\ref{thm_characterization_cs} using linear constraints, see Appendix \ref{sec_ilp_formulation}. Unfortunately, at least the present authors, can manage that only by using so-called Big-M-constraints, see e.~g.{} \cite{Koch2004b}. The consequence is that such a formulation works for a fixed number $n$ of voters only.

In order to be able to compute exact formulas depending on $n$ we exhaustively generate a finite list of sub-cases, which each can be modeled via a system of linear inequalities without using Big-M-constraints. Let us start with the conditions (a)(ii) and (a)(iv) of Theorem~\ref{thm_characterization_cs}. Since we have to ensure $\widetilde{m}_i\bowtie \widetilde{m}_j$ for all $i<j$, for each $i<j$ there must be integers $1\le a_{i,j}<b_{i,j}\le t$ such that
\begin{eqnarray}
  \sum_{h=1}^{k} m_{i,h}&=&\sum_{h=1}^{k} m_{j,h}\quad\forall 1\le k<a_{i,j},\\
  \sum_{h=1}^{a_{i,j}}m_{i,h}&>&\sum_{h=1}^{a_{i,j}} m_{j,h},\label{ie_bowtie_larger}\\
  \sum_{h=1}^{k} m_{i,h}&\ge&\sum_{h=1}^{k} m_{j,h}\quad\forall a_{i,j}<k<b_{i,j},\text{ and }\\
  \sum_{h=1}^{b_{i,j}}m_{i,h}&<&\sum_{h=1}^{b_{i,j}} m_{j,h}.
\end{eqnarray}
So, $a_{i,j}$ is the first index, where the partial sums of $\widetilde{m}_i$ and $\widetilde{m}_j$ differ. Due to condition
(a)(iv), they differ as in Inequality~(\ref{ie_bowtie_larger}). The next index with the partial sum of $\widetilde{m}_i$ smaller than the partial sum of $\widetilde{m}_j$ is denoted by $b_{i,j}$. In total there are ${t \choose 2}^{r\choose 2}$ possibilities for the $a_{i,j}$'s and the $b_{i,j}$'s. From (a)(iv), we conclude:
\begin{lemma}
  For $1\le i<j_1<j_2\le r$ we have $a_{i,j_1}\ge a_{i,j_2}$ and for
  $1\le i_1<i_2<j\le r$ we have $a_{i_1,j}\ge a_{i_2,j}$.
\end{lemma}

In order to ensure condition (a)(iii), assuming $t>1$, we introduce two integers $1\le c_j\le r$ and $0\le d_j<3^{c_j-1}$ for all $1\le j<t$. The interpretation of these integer should be the following: For a given column index $1\le j<t$, let $c_j$ be the first row index with
\begin{equation}
  m_{c_j,j}>0\quad\text{and}\quad m_{c_j,j+1}<n_{j+1}.
\end{equation}
Now, we consider the $d_j$ as numbers in base $3$, i.~e.{} we define $\hat{d}_{i,j}\in\{0,1,2\}$ for $1\le i<c_j$ via
\begin{equation}
  d_j=\sum\limits_{i=1}^{c_j-1} \hat{d}_{i,j}\cdot 3^{i-1}.
\end{equation}
To ensure that $c_j$ is the first row index, we require
\begin{eqnarray}
  m_{i,j}&=&0\,\,\,\quad\quad\text{if }\hat{d}_{i,j}\in\{0,2\},\\
  m_{i,j}&>&0\,\,\,\quad\quad\text{if }\hat{d}_{i,j}=1,\\
  m_{i,j+1}&=&n_{j+1}\quad\text{if }\hat{d}_{i,j}\in\{1,2\},\text{ and}\\
  m_{i,j+1}&<&n_{j+1}\quad\text{if }\hat{d}_{i,j}=0
\end{eqnarray}
for all $i<c_j$.

Thus, we obtain a finite list of sub-cases, which are described via the tuples
$$
  \Big(\left(a_{i,j}\right)_{1\le i<j\le r},\,\left(b_{i,j}\right)_{1\le i<j\le r},\,
  \left(c_j\right)_{1\le j<t},\,\left(d_j\right)_{1\le j<t}\Big).
$$
In each sub-case remains a system of linear inequalities having $(r+1)\times t$ variables
(without any reductions), where the number of voters $n$ occurs only on the right hand side,
whose number of integer solutions in principal can be determined using the parametric Barvinok algorithm.

Since the stated necessary conditions for the tuples of the sub-cases result in a huge number of
possibilities, even for small $t$ and $r$, we have to generate these tuples in a search tree where
we check whether the intermediate linear programs corresponding to the nodes of the search tree have a
solution or the requirements are contradicting, using \texttt{ILOG CPLEX  11.2}. 
Afterwards, we run the parametric Barvinok algorithm for each of the remaining possibilities,
see Table \ref{table_remaining_cases}.

\begin{table}[htp]
\begin{center}
\begin{tabular}{rrrrrrrr}
  \hline
  $\mathbf{(t,r)}$ & (3,2) & (3,3) & (3,4) & (3,5) & (4,2) & (4,3) & (5,2) \\
  \textbf{\#} & 9 & 46 & 254 & 1680 & 49 & 1071 & 217\\
  \hline
\end{tabular}

\medskip

\begin{tabular}{rrrrrrr}
  \hline
  $\mathbf{(t,r)}$ & (4,4) & (5,3) & (6,2) & (6,3) & (7,2) & (8,2) \\
  \textbf{\#} & 23666 & 17456 & 865 & 231081 & 3241 & 11665 \\
  \hline
\end{tabular}
\caption{Number of remaining cases for given $\mathbf{t}$, $\mathbf{r}$.}
\label{table_remaining_cases}
\end{center}
\end{table}

\vspace*{-5mm}

Using the described approach, we were able to determine exact formulas for $cs(n,t,r)$ in the cases where $r\le 4 $ for $t=3$,
$r\le 3$ for $t=4$, and $r=2$ for $t=5$, see Appendix \ref{sec_formulae}.

\section{Conclusion and open problems}
\label{sec_conclusion}

We have utilized a known parameterization of complete simple games in order to enumerate their number using the parametric Barvinok algorithm as a subroutine, which (so far) is indeed an obstructing bottleneck. Additionally, we have given a short proof for a surprising formula on the number of complete simple games with two types of voters. 

It would be interesting to discover an exact formula for $cs(n,3)$. Going over to weighted majority games, we can ask the same questions. In contrast to complete simple games, the number $wm(n,t)$ of weighted majority games with an arbitrary but fixed number $t$ types of voters is bounded by a polynomial in $n$, e.~g.{} one can easily show $wm(n,t)<(tn)^{t^4}$. Maybe also an exact enumeration formula for $wm(n,2)$ can be determined.



\bibliographystyle{abbrv}
\bibliography{dedekind}  
%
%


\appendix
\section{Complete simple games as integer points in polytopes}
\label{sec_ilp_formulation}
In the following, we will model the requirements of Theorem~\ref{thm_characterization_cs} using linear constraints and additional binary auxiliary variables. In order to obtain a bijection between complete simple games and lattice points, we have to ensure that for each complete simple games there exists exactly one allocation of the original and the auxiliary variables.
\begin{theorem}
  \label{thm_ilp_formulation}
  For given parameters $n$, $t$, and $r$ with $t+r>2$ each complete simple game attaining these parameters bijectively corresponds
  to a lattice point of a polytope $P$ defined by the following inequalities:
  \begin{eqnarray}
    n-\sum_{j=1}^t n_j &=&   0\label{ilp01}\\
    n_j-m_{i,j}          &\ge& 0\label{ilp02}\\
    \sum_{h=1}^j\left( m_{p,h} -m_{q,h}\right)-x_{p,q,j}^{(1)}+nx_{p,q,j}^{(2)} & \ge& 0\label{ilp03}\\
    \sum_{h=1}^j\left(m_{p,h} -m_{q,h}\right)-nx_{p,q,j}^{(1)} &\le& 0\label{ilp04}\\
    x_{p,q,j}^{(1)}+x_{p,q,j}^{(2)} &=& 1\label{ilp05}\\
    \sum_{j=1}^t x_{p,q,j}^{(1)} &\ge& 1\label{ilp06}\\
    m_{i,j'}-x_{i,j'}^{(3)} &\ge& 0\label{ilp07}\\
    k x_{i,j'}^{(3)} -m_{i,j'} &\ge& 0\label{ilp08}\\
    m_{i,j'+1}-n_{j'+1}+x_{i,j'}^{(4)} &\le& 0\label{ilp09}\\
    n_{j'+1}-m_{i,j'+1}-kx_{i,j'}^{(4)} &\le &0\label{ilp10}\\
    2x_{i,j'}^{(5)}-x_{i,j'}^{(3)}-x_{i,j'}^{(4)} &\le &0 \label{ilp11}\\
    x_{i,j'}^{(5)}-x_{i,j'}^{(3)}-x_{i,j'}^{(4)} &\ge& -1\label{ilp12}\\
    \sum_{i=1}^r x_{i,j'}^{(5)} &\ge& 1\label{ilp13}\\
    m_{i',j}-m_{i'+1,j}-x_{i',i'+1,j}^{(6)}+kx_{i',i'+1,j}^{(7)} &\ge& 0\label{ilp14}\\
    m_{i',j}-m_{i'+1,j}-k x_{i',i'+1,j}^{(6)}&\le& 0\label{ilp15}\\
    m_{i'+1,j}-m_{i',j}-x_{i'+1,i',j}^{(6)}+kx_{i'+1,i',j}^{(7)} &\ge& 0\label{ilp16}\\
     m_{i'+1,j}-m_{i',j}-k x_{i'+1,i',j}^{(6)}&\le& 0\label{ilp17}\\
    x_{i',i'+1,j}^{(6)}+x_{i',i'+1,j}^{(7)} &=& 1\label{ilp18}\\
    x_{i'+1,i',j}^{(6)}+x_{i'+1,i',j}^{(7)} &=& 1\label{ilp19}\\
    x_{i',j}^{(8)}-x_{i',i'+1,j}^{(6)} &\le& 0\label{ilp20}\\
    t(x_{i',j}^{(8)}-1)+\sum_{h=1}^j x_{i'+1,i',h}^{(6)} &\le& 0\label{ilp21}\\
    x_{i',j}^{(8)}-x_{i',i'+1,j}^{(6)}+\sum_{h=1}^j x_{i'+1,i',h}^{(6)} &\ge& 0\label{ilp22}\\
\end{eqnarray}
\begin{eqnarray}
    \sum_{j=1}^t x_{i',j}^{(8)} &\ge& 1\label{ilp23}\\
    n_j              &\in& \mathbb{N}_{>0}\label{ilp24}\\
    m_{i,j}          &\in& \mathbb{N}_0\label{ilp25}\\
    x_{p,q,j}^{(1)},x_{p,q,j}^{(2)} &\in& \{0,1\},\label{ilp26}\\
    x_{i,j'}^{(3)},x_{i,j'}^{(4)},x_{i,j'}^{(5)} &\in& \{0,1\},\label{ilp27}\\
    x_{i',i'+1,j}^{(6)},x_{i',i'+1,j}^{(7)},x_{i'+1,i',j}^{(6)},x_{i'+1,i',j}^{(7)} &\in& \{0,1\},\label{ilp28}\\
    x_{i',j}^{(8)} &\in &\{0,1\},\label{ilp29}
  \end{eqnarray}
  for all $1\le i\le r$, $1\le i'\le r-1$, $1\le j\le t$, $1\le j'\le t-1$, $1\le p,q\le r,p\neq q$, where
  $k=n-t+1$.
\end{theorem}
\begin{proof}
In order to prove the statement, on the one hand, we have to show that for fixed parameters $n$, $t$, $r$ and each complete simple game characterized by integers $n_j$ and $m_{i,j}$ there exists exactly one lattice point with coordinates $n_j$, $m_{i,j}$, $x_\star^{(\star)}$ in $P$. On the other hand, we have to show that for a given lattice point $n_j$, $m_{i,j}$, $x_\star^{(\star)}$ of $P$ the integers $n_j$, $m_{i,j}$ fulfill the conditions of a complete simple game.

At first we remark that $\widetilde{n}\in\mathbb{N}_{>0}$, $\sum\limits_{j=1}^t n_j=n$, and property~(i) of Theorem \ref{thm_characterization_cs} are equivalent to Inequalities (\ref{ilp01}), (\ref{ilp02}), (\ref{ilp24}), and (\ref{ilp25}).

The next step is to describe the interpretation of the auxiliary variables:
\begin{itemize}
\item[(a)] $x_{p,q,j}^{(1)}=1$ iff $\sum\limits_{h=1}^j m_{p,h}>\sum\limits_{h=1}^j m_{q,h}$, \quad $x_{p,q,j}^{(2)}=0$ iff $\sum\limits_{h=1}^j m_{p,h}>\sum\limits_{h=1}^j m_{q,h}$
\item[(b)] $x_{i,j}^{(3)}=1$ iff $m_{i,j}>0$
\item[(c)] $x_{i,j}^{(4)}=1$ iff $m_{i,j+1}<n_{j+1}$
\item[(d)] $x_{i,j}^{(5)}=1$ iff $m_{i,j}>0$ and $m_{i,j+1}<n_{j+1}$
\item[(e)] $x_{i,i+1,j}^{(6)}=1$ iff $m_{i,j}>m_{i+1,j}$, \quad $x_{i,i+1,j}^{(7)}=0$ iff $m_{i,j}>m_{i+1,j}$
\item[(f)] $x_{i+1,i,j}^{(6)}=1$ iff $m_{i+1,j}>m_{i,j}$, \quad  $x_{i+1,i,j}^{(7)}=0$ iff $m_{i+1,j}>m_{i,j}$
\item[(g)] $x_{i,j}^{(8)}=1$ iff $m_{i,j}>m_{i+1,j}$ and $m_{i,h}\ge m_{i+1,h}$ for all $1\le h\le j$.
\end{itemize}

Let us assume for a moment that these seven equivalences are valid. In this case the $x_\star^{(\star)}$ are uniquely determined by
$n_j$ and $m_{i,j}$ for $1\le i\le r$, $1\le j\le t$. Next we prove that properties~(ii)-(iv) of Theorem~\ref{thm_characterization_cs} are fulfilled for every lattice point $n_j$, $m_{i,j}$, $x_\star^{(\star)}$ in $P$.

Due to Inequality (\ref{ilp06}), for every $1\le p,q\le r$, $p\neq q$ there exists at least one index $j$ such that $x_{p,q,j}=1$. Using interpretation $(a)$, we can conclude that $\widetilde{m}_p\preceq \widetilde{m}_q$ does not hold. Interchanging $p$ and $q$ yields $\widetilde{m}_p\bowtie \widetilde{m}_q$ for all $p\neq q$, which is equivalent to property~(ii) of Theorem \ref{thm_characterization_cs}.

If $t=1$ then due to the required $t+r>2$ we have $r\ge 2$ and property~(iv) of Theorem \ref{thm_characterization_cs} implies
$m_{1,1}>0$, which is equivalent to property~(iii) of Theorem \ref{thm_characterization_cs} in this case. In the remaining cases ($t>1$)  due to Inequality~(\ref{ilp13}), for every $1\le j<t$ there exists at least one row-index $i$ such that $x_{i,j}^{(5)}=1$. Using interpretation~(d), we can conclude $m_{i,j}>0$ and $m_{i,j+1}<n_{j+1}$. Thus the $n_j$, $m_{i,j}$ fulfill property~(iii)
of Theorem \ref{thm_characterization_cs}.

Due to Inequality (\ref{ilp23}), for each $1\le i\le r-1$ there is at least one $1\le j\le t$ such that $x_{i,j}^{(8)}=1$. Using interpretation (g), we conclude $\widetilde{m}_{i}\gtrdot \widetilde{m}_{i+1}$ which is equivalent to property~(iv) of Theorem \ref{thm_characterization_cs}.

\bigskip

All that remains to prove are the seven interpretations. Obviously we have $$\left|\sum_{h=1}^j\left(m_{p,h} -m_{q,h}\right)\right|\le n,$$
$$\left|m_{i,j}-m_{i+1,j}\right|\le m_{i,j}\le n-t+1,$$
and
$$n_j-m_{i,j}\le n-t+1.$$
\begin{itemize}
\item[(a)] If $\sum\limits_{h=1}^j m_{p,h}>\sum\limits_{h=1}^j m_{q,h}$ for given $p$, $q$, and $j$, then due to inequalities
(\ref{ilp04}) and (\ref{ilp05}) we have $x_{p,q,j}^{(1)}=1$ and $x_{p,q,j}^{(2)}=0$. In this case Inequality
(\ref{ilp03}) is also valid. Otherwise if $\sum\limits_{h=1}^j m_{p,h}\le\sum\limits_{h=1}^j m_{q,h}$ for given $p$, $q$,
and $j$, then due to inequality (\ref{ilp03}) and (\ref{ilp05}) we have $x_{p,q,j}^{(1)}=0$ and $x_{p,q,j}^{(2)}=1$.
In this case Inequality (\ref{ilp04}) is valid, too.
\item[(b)] If $m_{i,j}>0$ then due to Inequality (\ref{ilp08}) we have $x_{i,j}^{(3)}=1$ so that Inequality (\ref{ilp07}) is also
fulfilled. If $m_{i,j}=0$ then due to Inequality (\ref{ilp07}) we have $x_{i,j}^{(3)}=0$ so that Inequality (\ref{ilp08})
is also valid.
\item[(c)] If $m_{i,j+1}<n_{j+1}$ then due to Inequality (\ref{ilp10}) we have $x_{i,j}^{(4)}=1$. If $m_{i,j+1}=n_{j+1}$ then
due to Inequality (\ref{ilp09}) we have $x_{i,j}^{(4)}=0$. In both cases, inequalities (\ref{ilp09}) and (\ref{ilp10})
are valid.
\item[(d)] If $m_{i,j}>0$ and $m_{i,j+1}<n_{j+1}$ we have $x_{i,j}^{(3)}=1$ and $x_{i,j}^{(4)}=1$. Due to Inequality
(\ref{ilp12}) we have $x_{i,j}^{(5)}=1$ in this case. Otherwise we have $x_{i,j}^{(3)}+x_{i,j}^{(4)}\le 1$ and conclude
$x_{i,j}^{(5)}=0$ from Inequality (\ref{ilp11}). In both cases inequalities (\ref{ilp11}) and (\ref{ilp12}) are valid.
\item[(e)] If $m_{i,j}>m_{i+1,j}$ then, due to inequalities (\ref{ilp15}) and (\ref{ilp18}), we have $x_{i,i+1,j}^{(6)}=1$ and
$x_{i,i+1,j}^{(7)}=0$. Otherwise we have $m_{i,j}\le m_{i+1,j}$ and conclude $x_{i,i+1,j}^{(6)}=0$ and
$x_{i,i+1,j}^{(7)}=1$ from inequalities (\ref{ilp14}) and (\ref{ilp18}). In both cases inequalities (\ref{ilp14}) and
(\ref{ilp15}) are valid.
\item[(f)] Similar to (e).
\item[(g)] If $m_{i,j}>m_{i+1,j}$ and $m_{i,h}\ge m_{i+1,h}$ for all $1\le h\le j$ then we have $x_{i,i+1,j}^{(6)}=1$,
$x_{i,i+1,j}^{(7)}=0$ and $x_{i+1,i,h}^{(6)}=0$, $x_{i+1,i,h}^{(7)}=1$ for all $1\le h\le j$. Thus we can conclude
$x_{i,j}^{(8)}=1$ from Inequality (\ref{ilp22}). If $m_{i,j}\le m_{i+1,j}$ then we have $x_{i,i+1,j}^{(6)}=0$ so that
Inequality (\ref{ilp20}) forces $x_{i,j}^{(8)}=0$. If there exists an index $1\le h\le t$ such that $m_{i,h}< m_{i+1,h}$
then we have $x_{i+1,i,h}^{(6)}=1$ and conclude $x_{i,j}^{(8)}=0$ from Inequality (\ref{ilp21}). In all cases, inequalities
(\ref{ilp20}), (\ref{ilp21}), and (\ref{ilp22}) are valid.
\end{itemize}
\end{proof}

\section{Formulas for the number of complete simple games}
\label{sec_formulae}

\begin{lemma}
  $$
    cs(n,1,1)=cs(n,1)=n.
  $$
\end{lemma}

\begin{lemma}
  For $t\ge2$ we have
  $$
    cs(n,t,1)={n+1\choose 2t-1}.
  $$
\end{lemma}


Using the represented approach from Section~\ref{sec_main} we were able to determine the following formulas for $cs(n,t,r)$:

\medskip

\noindent
$cs(n,3,2)=0$ for $n\le 3$ and $cs(n,3,2)=$
\begin{eqnarray*}
  \!\!\!\!\!\!\!\!\!\!\!\!&&
  \frac{1}{26880}n^8 + \frac{13}{20160}n^7 -\frac{1}{2880}n^6-\frac{43}{5760}n^5-\frac{1}{2880}n^4\\
  \!\!\!\!\!\!\!\!\!\!\!\!&&
  + \frac{23}{1440}n^3 + \frac{23}{5040}n^2 + \left[\frac{1}{70},-\frac{41}{4480}\right]_n\cdot n +
  \left[\frac{0}{1},-\frac{1}{256}\right]_n
\end{eqnarray*}
for $n\ge 2$.

\bigskip
\bigskip

\noindent
$cs(n,3,3)=0$ for $n\le 4$ and $cs(n,3,3)=$
\begin{eqnarray*}
  \!\!\!\!\!\!\!\!\!\!\!\!&&
  \frac{23}{239500800}n^{11} + \frac{139}{87091200}n^{10} + \frac{257}{52254720}n^9 -\frac{107}{1161216}n^8\\
  \!\!\!\!\!\!\!\!\!\!\!\!&&
  + \frac{871}{14515200}n^7 + \frac{1177}{1555200}n^6 -\frac{1571}{1088640}n^5 -\frac{5}{6804}n^4 \\
  \!\!\!\!\!\!\!\!\!\!\!\!&&
  +\left[\frac{1429}{302400},\frac{21289}{4838400}\right]_n\cdot n^3 +
  \left[-\frac{401}{151200},\frac{16861}{9676800}\right]_n\cdot n^2\\
  \!\!\!\!\!\!\!\!\!\!\!\!&&
  + \left[-\frac{1}{3080},-\frac{10393}{2365440}\right]_n\cdot n\\
  \!\!\!\!\!\!\!\!\!\!\!\!&&
  + \left[\frac{0}{1},-\frac{451}{1492992},-\frac{4}{729},\frac{5}{2048},-\frac{2}{729},-\frac{4547}{1492992}\right]_n
\end{eqnarray*}
for $n\ge 2$.

\bigskip
\bigskip

\noindent
$cs(n,3,4)=0$ for $n\le 5$ and $cs(n,3,4)=$
\begin{eqnarray*}
  \!\!\!\!\!\!\!\!\!\!\!\!&&
  \frac{2833}{16738231910400}n^{14} + \frac{913}{391283343360}n^{13} -\frac{25733}{3310859059200}n^{12}\\
  \!\!\!\!\!\!\!\!\!\!\!\!&&
  -\frac{8329}{41385738240}n^{11} -\frac{104849}{300987187200}n^{10} + \frac{434377}{30098718720}n^9\\
  \!\!\!\!\!\!\!\!\!\!\!\!&&
  -\frac{3853}{85730400}n^8  -\frac{56471}{752467968}n^7 + \frac{10222451}{18811699200}n^6\\
  \!\!\!\!\!\!\!\!\!\!\!\!&&
  + \left[-\frac{103807}{209018880},-\frac{1628593}{3344302080}\right]_n\cdot n^5\\
  \!\!\!\!\!\!\!\!\!\!\!\!&&
  + \left[-\frac{3612949}{2874009600},-\frac{418954397}{367873228800}\right]_n\cdot n^4\\
  \!\!\!\!\!\!\!\!\!\!\!\!&&
  + \left[\frac{3217}{1451520},\frac{33517}{23224320}\right]_n\cdot n^3
  + \Big[-\frac{913147}{1816214400},\\
  \!\!\!\!\!\!\!\!\!\!\!\!&&
  \frac{484043734439}{301288174387200},-\frac{29120107}{147113366400},\frac{5408921719}{3719607091200},\\
  \!\!\!\!\!\!\!\!\!\!\!\!&&
  -\frac{51542507}{147113366400},\frac{29964809639}{301288174387200}\Big]_n\cdot n^2
  + \Big[-\frac{197}{144144},\\
  \!\!\!\!\!\!\!\!\!\!\!\!&&
  -\frac{719771827}{478235197440},\frac{80611}{105080976},-\frac{16985027}{5904138240},\frac{59}{11675664},\\
  \!\!\!\!\!\!\!\!\!\!\!\!&&
  -\frac{3197869643}{4304116776960}\Big]_n\cdot n + \Big[\frac{0}{1},-\frac{306295}{859963392},-\frac{2}{729},\frac{233}{131072},\\
  \!\!\!\!\!\!\!\!\!\!\!\!&&
  -\frac{14}{6561},-\frac{92287}{95551488}\Big]_n
\end{eqnarray*}
for $n\ge 2$.

\bigskip
\bigskip

\noindent
$cs(n,4,2)=0$ for $n\le 4$ and $cs(n,4,2)=$
\begin{eqnarray*}
  \!\!\!\!\!\!\!\!\!\!\!\!&&
  \frac{29}{319334400}n^{11} + \frac{197}{58060800}n^{10} + \frac{67}{3870720}n^9 -\frac{793}{3870720}n^8\\
  \!\!\!\!\!\!\!\!\!\!\!\!&&
  -\frac{341}{9676800}n^7 + \frac{667}{345600}n^6 -\frac{73}{45360}n^5 -\frac{2791}{725760}n^4\\
  \!\!\!\!\!\!\!\!\!\!\!\!&&
  + \frac{2683}{604800}n^3 + \left[-\frac{41}{12600},\frac{37283}{6451200}\right]_n\cdot n^2\\
  \!\!\!\!\!\!\!\!\!\!\!\!&&
  + \left[\frac{53}{9240},-\frac{13289}{4730880}\right]_n\cdot n + \left[\frac{0}{1},-\frac{15}{4096}\right]_n
\end{eqnarray*}
for $n\ge 2$.

\bigskip
\bigskip

\noindent
$cs(n,4,3)=0$ for $n\le 4$ and $cs(n,4,3)=$
\begin{eqnarray*}
  \!\!\!\!\!\!\!\!\!\!\!\!&&
  \frac{40441/}{1506440871936000}n^{15} + \frac{227}{190749081600}n^{14}\\
  \!\!\!\!\!\!\!\!\!\!\!\!&&
  + \frac{55447}{2459495301120}n^{13} + \frac{731}{5748019200}n^{12}-\frac{2450159}{3678732288000}n^{11}\\
  \!\!\!\!\!\!\!\!\!\!\!\!&&
  -\frac{494953}{66886041600}n^{10}+ \frac{9002453}{316036546560}n^9 -\frac{4213}{8670412800}n^8\\
  \!\!\!\!\!\!\!\!\!\!\!\!&&
  -\frac{36098869}{188116992000}n^7 + \left[\frac{623521}{836075520},\frac{9976903}{13377208320}\right]_n\cdot n^6\\
  \!\!\!\!\!\!\!\!\!\!\!\!&&
  + \left[-\frac{63029}{574801920},-\frac{16678043}{147149291520}\right]_n\cdot n^5\\
  \!\!\!\!\!\!\!\!\!\!\!\!&&
  + \left[-\frac{1341127}{479001600},-\frac{157994831}{61312204800}\right]_n\cdot n^4\\
  \!\!\!\!\!\!\!\!\!\!\!\!&&
  + \left[\frac{5098711}{18162144000},-\frac{34324686997}{37196070912000}\right]_n\cdot n^3\\
  \!\!\!\!\!\!\!\!\!\!\!\!&&
  + \left[\frac{12701}{46569600},\frac{716763911}{190749081600}\right]_n\cdot n^2 + \Big[\frac{1711}{720720},\\
  \!\!\!\!\!\!\!\!\!\!\!\!&&
  \frac{1108675}{1721646710784},\frac{3809879}{525404880},-\frac{5037829}{2361655296},\frac{2368439}{525404880},\\
  \!\!\!\!\!\!\!\!\!\!\!\!&&
  \frac{4724419267}{1721646710784}\Big]_n\cdot n + \Big[\frac{0}{1},-\frac{3711109}{5159780352},\frac{21277}{20155392},\\
  \!\!\!\!\!\!\!\!\!\!\!\!&&
  -\frac{631}{262144},\frac{14}{19683},\frac{3104635}{5159780352},-\frac{1}{1024},-\frac{8749957}{5159780352},\\
  \!\!\!\!\!\!\!\!\!\!\!\!&&
  \frac{40}{19683},-\frac{375}{262144},-\frac{5347}{20155392},-\frac{1934213}{5159780352}\Big]_n
\end{eqnarray*}
for $n\ge 2$.

\bigskip
\bigskip

\noindent
$cs(n,5,2)=0$ for $n\le 5$ and $cs(n,5,2)=$
\begin{eqnarray*}
  \!\!\!\!\!\!\!\!\!\!\!\!&&
  \frac{1}{10729635840}n^{14} + \frac{37}{6642155520}n^{13} + \frac{1}{12773376}n^{12}\\
  \!\!\!\!\!\!\!\!\!\!\!\!&&
  -\frac{13}{24330240}n^{11} -\frac{9683}{1393459200}n^{10} + \frac{1637}{30965760}n^9\\
  \!\!\!\!\!\!\!\!\!\!\!\!&&
  -\frac{45539}{975421440}n^8-\frac{4961}{11612160}n^7 +\frac{2741}{2419200}n^6 -\frac{23}{92160}n^5\\
  \!\!\!\!\!\!\!\!\!\!\!\!&&
  -\frac{208643}{95800320}n^4 + \left[\frac{5293}{1330560},\frac{272099}{170311680}\right]_n\cdot n^3\\
  \!\!\!\!\!\!\!\!\!\!\!\!&&
  + \left[-\frac{110147}{23284800},\frac{35652461}{11921817600}\right]_n\cdot n^2\\
  \!\!\!\!\!\!\!\!\!\!\!\!&&
  + \left[\frac{29}{12012},-\frac{95765}{98402304}\right]_n\cdot n + \left[\frac{0}{1},-\frac{31}{16384}\right]_n
\end{eqnarray*}
for $n\ge 2$.

We would like to remark, that we have numerically veryfied the stated formulas for all $n\le 11$ via exhaustive generation.

\section{Number of sub-cases for the tuples from Section 5}
\label{sec_sub_cases}


For small values of $t$ and $r$ we have generated all feasible tuples
$$
  \Big(\left(a_{i,j}\right)_{1\le i<j\le r},\,\left(b_{i,j}\right)_{1\le i<j\le r},\,
  \left(c_j\right)_{1\le j<t},\,\left(d_j\right)_{1\le j<t}\Big).
$$
in a tree search, where we have solved the corresponding intermediate integer programs
using \texttt{ILOG CPLEX 11.2}. In the third column of Table~\ref{table_all_sub_cases}
we state their number and in the fourth column the computation time for the search. Here
we have to remark that the implementation of this search is highly non-optimized since
the bottleneck of the overall algorithm lies in the application of the parametric Barvinok
algorithm. To get an impression of the latter fact we have given the (known) running times of
the parametric Barvinok algorithm on the special sub-case $a_{i,j}=1$, $b_{i,j}=2$, $c_j=1$,
$d_j=0$, the running time of the non-parametric Barvinok algorithm on the same sub-case with $n=20$,
and the total running time for the parametric Barvinok algorithm in the last three columns of
Table~\ref{table_all_sub_cases}.

\begin{table}[htp]
  \begin{center}
    \begin{tabular}{|rrr|r|rrr|}
      \hline
      $\mathbf{t}$ & $\mathbf{r}$ & $\mathbf{\#}$ & \textbf{cases} & \multicolumn{3}{c}{\textbf{barvinok}}\\
       3 & 2 &        9 &    0s &   0s &  0s &  1s \\
       3 & 3 &       46 &    1s &  19s &  0s &  6m \\
       3 & 4 &      254 &   14s &  43m &  0s & 91h \\
       3 & 5 &     1680 &    5m &  52h &  1s &     \\
       3 & 6 &    13474 &  123m &      &  5s &     \\
       4 & 2 &       49 &    0s &  14s &  0s &  2m \\
       4 & 3 &     1071 &   13s & 222m &  0s & 13d \\
       4 & 4 &    23666 &   12m &      &  5s &     \\
       5 & 2 &      217 &    1s &  27m &  0s &  5h \\
       5 & 3 &    17456 &    3m &      &  4s &     \\
       6 & 2 &      865 &    4s &      &  0s &     \\
       6 & 3 &   231081 &   31m &      &  6m &     \\
       7 & 2 &     3241 &   14s &      &  2s &     \\
       7 & 3 &  2679286 &  320m &      &     &     \\
       8 & 2 &    11665 &   51s &      &  9s &     \\
       9 & 2 &    40825 &    3m &      &  2m &     \\
      10 & 2 &   139969 &   10m &      & 62m &     \\
      11 & 2 &   472393 &   36m &      &     &     \\
      12 & 2 &  1574641 &  123m &      &     &     \\
      13 & 2 &  5196313 &  364m &      &     &     \\
      14 & 2 & 17006113 & 1150m &      &     &     \\
      \hline
    \end{tabular}
    \caption{Number of all sub-cases for given parameters $\mathbf{t}$ and $\mathbf{r}$.}
    \label{table_all_sub_cases}
  \end{center}
\end{table}

In order to determine some more exact formulas for the number $cs(n,t,r)$ of complete simple games one should exploit the fact that the polytopes of the sub-cases
are related, i.~e.{} they differ only in a few describing hyperplanes. It should be possible to re-implement the parametric Barvinok algorithm for sets of similar polytopes, reusing common intermediate results. Comparing the timing results of the second last column with the third last column of Table~\ref{table_all_sub_cases}, it may be possible to deduce the special types of the counting quasi-polynomials in each sub-case in order to determine them using interpolation and several runs of the non-parametric Barvinok algorithm. Here we would like to remark, that the used implementation in some cases determines large quasi-polynomials which are only valid for a small number, e.~g.{} $3$, of $n$-values.

So finally we have to say, that our approach for the determination of exact formulas for $cs(n,t,r)$, in dependence of the number of voters $n$, looks promising, but a lot of research and implementational work may be needed in order to determine some more formulas.

\section{Numerical data on the number of complete simple games}
\label{sec_numerical}

For $t\in\{3,4\}$ we were able to compute some further exact values of $cs(n,t)$, see Table~\ref{table_cs_3} and Table~\ref{table_cs_4}.

\begin{table}[htp]
  \begin{center}
    \begin{tabular}{rr}
      \hline
      $\mathbf{n}$ & $\mathbf{cs(n,3)}$\\
       1 &             0 \\
       2 &             0 \\
       3 &             0 \\
       4 &             6 \\
       5 &            50 \\
       6 &           262 \\
       7 &          1114 \\
       8 &          4278 \\
       9 &         15769 \\
      10 &         58147 \\
      11 &        221089 \\
      12 &        886411 \\
      13 &       3806475 \\
      14 &      17681979 \\
      15 &      89337562 \\
      16 &     492188528 \\
      17 &    2959459154 \\
      18 &   19424078142 \\
      19 &  139141985438 \\
      20 & 1087614361775 \\
      21 & 9274721292503 \\
      \hline
    \end{tabular}
    \caption{Known values of the number $\mathbf{cs(n,3)}$ of complete simple games with $\mathbf{3}$ types of voters.}
    \label{table_cs_3}
  \end{center}
\end{table}

\begin{table}[htp]
  \begin{center}
    \begin{tabular}{rr}
      \hline
      $\mathbf{n}$ & $\mathbf{cs(n,4)}$\\
       1 &             0 \\
       2 &             0 \\
       3 &             0 \\
       4 &             0 \\
       5 &            24 \\
       6 &           426 \\
       7 &          4769 \\
       8 &         45483 \\
       9 &        431440 \\
      10 &       4570902 \\
      11 &      59776637 \\
      12 &    1047858496 \\
      13 &   26000281487 \\
      \hline
    \end{tabular}  
    \caption{Known values of the number $\mathbf{cs(n,4)}$ of complete simple games with $\mathbf{4}$ types of voters.}
    \label{table_cs_4}
  \end{center}
\end{table}

\end{document}